\documentclass[11pt]{amsart}
\usepackage[left=1.25in,right=1.25in]{geometry}
\usepackage{amsthm}
\usepackage{amsmath}
\usepackage{mathtools}
\usepackage{amssymb}
\usepackage{hyperref}
\usepackage{tikz-cd}
\usepackage{tikz}
\usepackage{array}
\usetikzlibrary{decorations.markings}
\usepackage{enumitem}
\usepackage{amsfonts}
\usepackage{hyperref}
\usepackage{todonotes}
\usepackage[utf8]{inputenc}
\usepackage{ytableau}
\usepackage{subcaption}
\usepackage[export]{adjustbox}
\usepackage{mathdots}
\usepackage{comment}
\usepackage{float}
\usepackage{bbm}

\newtheorem{theorem}{Theorem}[section]
\newtheorem{prop}[theorem]{Proposition}
\newtheorem{conj}[theorem]{Conjecture}
\newtheorem{lemma}[theorem]{Lemma}
\newtheorem{cor}[theorem]{Corollary}

\theoremstyle{definition}
\newtheorem{ex}[theorem]{Example}

\newtheorem{defin}[theorem]{Definition}
\newtheorem{remark}[theorem]{Remark}

\makeatletter
\newcommand\circled[1]{%
  \mathpalette\@circled{#1}%
}
\newcommand\@circled[2]{%
  \tikz[baseline=(math.base)] \node[draw,circle,inner sep=1pt] (math) {$\m@th#1#2$};%
}
\makeatother

\title{Bounds on Schubert coefficients in the two-row case}

\author{Zijie Tao}
\address{Department of Mathematics, Peking University}
\email{\href{mailto:2100012931@stu.pku.edu.cn}{{\tt 2100012931@stu.pku.edu.cn}}}

\author{Yunchi Zheng}
\address{Department of Mathematics, Nankai University}
\email{\href{mailto:2111565@mail.nankai.edu.cn}{{\tt 2111565@mail.nankai.edu.cn}}}
\date{\today}

\begin{document}
\begin{abstract}
We provide an upper bound for generalized Littlewood-Richardson coefficients $c_{uv}^w$, where $u$ is a two-row Young diagram corresponding to a Grassmannian permutation. We end with a conjecture on the upper bounds for all such structure constants.
\end{abstract}
\maketitle

\section{Introduction}\label{sec:intro}

The \emph{generalized Littlewood-Richardson coefficients} (also called \emph{Schubert structure constants}) are crucial in the theory of Schubert calculus, as they are the structure constants of the cohomology ring of the full flag variety with respect to the Schubert basis. Their importance lies in the center of Schubert calculus and its relation with combinatorics and geometry. For the special case of $s_\lambda s_\mu$ where $s_\lambda$ and $s_\mu$ are Schur polynomials, Littlewood and Richardson give a combinatorial construction for the structure coefficients. In 1996, Sottile \cite{sottile1996pieri} gave a Pieri's formula for calculating the product of a Schubert polynomial with the elementary symmetric polynomial and complete homogeneous polynomial. Later, Fomin and Kirillov~\cite{fomin1999quadratic} constructed a quadratic algebra, which is now called \emph{Fomin-Kirillov algebra}, and revealed that it has profound connections with Schubert calculus and quantum cohomology. Also, Postnikov~\cite{postnikov1999quantum} discovered a new Pieri formula in the quantum version by virtue of Fomin-Kirillov algebra. Using this \emph{quantum Pieri's formula}, Mészáros, Panova and Postnikov~\cite{meszaros2012schur} gave a solution to Fomin-Kirillov nonnegativity conjecture when the Schur polynomial is of hook shape. In \cite{buch2014mutationspuzzlesequivariantcohomology}, 
Buch introduced a mutation algorithm for puzzles that is a three-direction analogue of the classical jeu de taquin algorithm for semistandard tableaux. He applies this algorithm to establish puzzle formula for the equivariant Schubert structure constants of two-step flag varieties. This formula gives an expression for the structure constants that is positive in the sense of Graham. Recently, Daoji Huang~\cite{Huang_2022} provided a counting formula in terms of reduced word tableaux, for computing the structure constants of products of Schubert polynomials indexed by permutations with separated descents, and recognized that these structure constants are certain Edelman-Greene coefficients. 

Now we introduce our main results as follows. Let $S_n$ be the symmetric group of permutations on $[n]$. For any $w \in S_{n},$ let $\mathfrak{S}_{w}$ be the Schubert polynomial associated to $w$ as defined in Definition~\ref{D:Schubert polynomial}. Let $\lambda= (\lambda_{1}\geq \lambda_{2}\geq\cdots \geq \lambda_{d})$ be a partition of $n$ and let $s_{\lambda}$ be the Schur polynomial of $d$ variables associated to $\lambda$ as defined in Definition~\ref{D:Schubert polynomial}. Note that Schur polynomials are Schubert polynomials for Grassmannian permutations as defined in Definition~\ref{Grassmannian}. In the expansion
\[\mathfrak{S}_{u}\mathfrak{S}_{v}=\sum\limits_{w\in S_{n}}c_{uv}^{w}\mathfrak{S}_{w},
\]
$c_{uv}^w$ are called the \emph{generalized Littlewood-Richardson coefficients}. The following is our main result:
\begin{theorem}
\label{I: Main theorem 1}
    Let $n_1,n_2$ be two positive integers and $n=\max\{n_1,n_2\}$. Let $\tau$ be a $k$-Grassmannian permutation in $S_{n_1}.$ Suppose $\lambda$ is the Young diagram with $2$ rows corresponding to $\tau.$ Let $w$ be a permutation in $S_{n_2}.$  Write $\mathfrak{S}_{w}s_{\lambda}(x_1,\cdots,x_k)=\sum\limits_{v\in S_{n}}c_{w\tau}^{v}\mathfrak{S}_{v}$. Then
    \begin{enumerate}
        \item $c_{w\tau}^{v}\in\{0,1\}$ if $k=n_2$.
        \item $c_{w\tau}^{v}\in\{0,1\}$ if $k=n_2-1$.
        \item $c_{w\tau}^{v}\in \{0,1,2\}$ if $k=n_2-2$.
    \end{enumerate} 
\end{theorem}

For $k$ in general positions, we have the following theorem:
\begin{theorem}
\label{I: Main theorem 2}
    Let $n_1,n_2$ be two positive integers and $n=\max\{n_1,n_2\}$. Let $\tau$ be a $k$-Grassmannian permutation in $S_{n_1}.$ Suppose $\lambda$ is the Young diagram with $2$ rows corresponding to $\tau$ and $\lambda=(m_1,m_2).$ Let $w$ be a permutation in $S_{n_2}.$ Write $\mathfrak{S}_{w}s_{\lambda}(x_1,\cdots,x_k)=\sum\limits_{v\in S_{n}}c_{ w\tau}^{v}\mathfrak{S}_{v}$. If $n_2-k<m_2,$ then $c_{w\tau}^{v}\leq 2^{n_2-k}(n_2-k)!.$ If we further assume that $w(k+1)>w(k+2)>\cdots w(n_2),$ we have $c_{w\tau}^{v}\leq 2^{n_2-k}.$
\end{theorem}

In fact, we formulate the following conjecture:
\begin{conj}
\label{I: Conj 3}
    Let $n_1,n_2$ be two positive integers and $n=\max\{n_1,n_2\}$. Let $\tau$ be a $k$-Grassmannian permutation in $S_{n_1}.$ Suppose $\lambda$ is the Young diagram with $2$ rows corresponding to $\tau.$ Let $w$ be a permutation in $S_{n_2}.$ Write $\mathfrak{S}_{w}s_{\lambda}(x_1,\cdots,x_k)=\sum\limits_{v\in S_{n}}c_{ w\tau}^{v}\mathfrak{S}_{v}$. We have $0\leq c_{w \tau}^{v}\leq n_2-k.$
\end{conj}

We briefly outline the structure of this paper. In Section 2, we review some basic knowledge of classical Schubert calculus. Section 3 is devoted to the proof of Theorem~\ref{I: Main theorem 1}. Section 4 is devoted to the proof of Theorem~\ref{I: Main theorem 2}. We will discuss conjecture~\ref{I: Conj 3} at the end of the paper.

\section{Definitions and Preliminaries}\label{sec:prelim}

In this section, we review some basic knowledge of classical Schubert calculus. We first introduce some definitions:
\begin{defin}
\label{D:Young diagram}
For any positive integer $n,$ a \emph{Young diagram} with $n$ boxes is a partition $\lambda$ of $n,$ which is a sequence of weakly decreasing positive integers \[(\lambda_{1}\geq\lambda_{2}\geq\cdots\geq\lambda_{d}),\] such that $\sum\limits_{i=1}^{d}\lambda_{i}=n.$ A \emph{semistandard Young tableau} $T$ of shape $\lambda$ is an array $T$ of positive integers of the form $(T_{i,j}\in \mathbb{N}_{+})_{1\leq i\leq d;1\leq j\leq \lambda_{i}}$ such that 
\begin{enumerate}
    \item for any fixed $i$ and $j_{1}<j_{2},$ we have $T_{i,j_{1}}\leq T_{i,j_{2}};$
    \item for any fixed $j$ and $i_{1}<i_{2},$ we have $T_{i_{1},j}\leq T_{i_{2},j}.$
\end{enumerate}
\end{defin}
\begin{defin}
\label{D:Schur polynomial}
For any positive integer $n$, let $\lambda=(\lambda_1 \geq \lambda_2 \geq \cdots \geq \lambda_d)$ be a partition of $n$ with $d$ parts, i.e., $\sum\limits_{i=1}^{d}\lambda_{i}=n.$ Let $k$ be a positive integer greater than or equal to $d$. We define the \emph{Schur polynomial} of $\lambda$ on $k$ variables as 
\begin{equation}
    s_\lambda(x_{1},x_{2},\cdots,x_{k})=\sum\limits_T x^T=\sum\limits_T x_{1}^{t_1}\cdots x_{k}^{t_k}
\end{equation}
summing over all semistandard Young tableaux $T$ of shape $\lambda$. where $t_i$ counts the occurrences of the number $i$ in $T.$
In particular, the \emph{elementary homogeneous polynomial} and the \emph{completely homogeneous polynomial} of degree $n$ with $k$ variables for $k\geq n$ are $e_n=s_{1^n}(x_{1},x_{2},\cdots,x_{k})$ and $h_n=s_{n}(x_{1},x_{2},\cdots, x_{k}),$ respectively.
\end{defin}
Now we give the definition of Schubert polynomials:
\begin{defin}
\label{D:Schubert polynomial}
For a permutation $w\in S_n$, we define the \emph{Schubert polynomial} as follows:
\begin{itemize}
    \item If $w=w_{0},$ where $w_0$ is the longest permutation in $S_{n}$, then $\mathfrak{S}_{w_0}=x_1^{n}x_2^{n-1}\cdots x_n$
    \item Otherwise, $\partial_{i}\mathfrak{S}_{w}=\mathfrak{S}_{ws_i}$ if $w(i)>w(i+1).$
\end{itemize}
Here the divided difference operator $\partial_{i}$ is defined such that for any polynomial $f,$ we have $\partial_{i} f=\frac{f-s_{i}f}{x_{i}-x_{i+1}}$ with $s_{i}\in S_n$ acting on $f$ by swapping $x_i$ and $x_{i+1}.$

\end{defin}
\begin{remark}
\label{greek&english}
To distinguish between Young diagram and permutations of $S_n$, we use Greek letters to represent Young diagram and English letters for permutations.
\end{remark}
\begin{defin}
\label{Grassmannian}
    Let $k,n$ be two integers. A permutation $w\in S_n$ is called \emph{$k$-Grassmannian} if $w(1)<\cdots<w(k),w(k+1)<\cdots<w(n).$ 
\end{defin}
\begin{remark}
    \label{D:connection}
    Schur polynomials can be viewed as a special case of Schubert polynomials. More explicitly, for a Schur polynomial $s_{\lambda}$ of with $k$ variables corresponding to a partition $\lambda$, there exists a positive integer $n$ such that $s_{\lambda}=\mathfrak{S}_{w},$ where $w\in S_n$ is a $k$-Grassmannian permutation satisfying $w(1)<\cdots<w(k)$, $w(k+1)<\cdots<w(n)$. 
    
    Conversely, for a $k$-Grassmannian permutation $w\in S_n,$ then $\mathfrak{S}_{w}$ is the Schur polynomial $s_{\lambda}(x_1,\cdots,x_k)$ where $\lambda$ is the partition $(w(k)-k,\cdots,w(2)-2,w(1)-1).$
\end{remark}
\begin{ex}
\label{D:example 2.5}
    Let $n=5,k=3.$ The \emph{elementary homogeneous polynomial} $e_{3}(x_1,x_2,x_3)=x_1x_2x_3=\mathfrak{S}_{(2,3,4,1,5)}$ and the \emph{completely homogeneous polynomial} $h_2(x_1,x_2,x_3)=x_1^2+x_2^2+x_3^2+x_1x_2+x_1x_3+x_2x_3=\mathfrak{S}_{(1,2,5,3,4)}.$
\end{ex}

For any $u,v\in S_{n},$ we have
\[\mathfrak{S}_{u}\mathfrak{S}_{v}=\sum\limits_{w\in S_{n}}c_{uv}^{w}\mathfrak{S}_{w}.
\]
We call $c_{uv}^{w}$ \emph{the generalized Littlewood-Richardson coefficients}, which are nonnegative.

Moreover, we have the \emph{Jacobi-Trudi formula}:
\begin{lemma}\label{D:Jacobi-Trudi}
Let $\lambda=(m_1,m_2)$ be a partition of $n$ with two rows and $k$ be a positive integer larger than $1$, then the Schur polynomial of shape $\lambda$ with $k$ variables satisfies 
\[
s_\lambda(x_{1},x_{2},\cdots,x_{k})=h_{m_1}h_{m_2}-h_{m_1+1}h_{m_2-1}.
\]
Here $h_{m}$ is the completely homogeneous symmetric polynomial of degree $m$ with $k$ variables. 
\end{lemma} 

For the product of a \emph{completely homogeneous symmetric polynomial} with a \emph{Schubert polynomial}, Sottile gave the following lemma as in \cite[Theorem 1]{sottile1996pieri}:
\begin{lemma}
\label{D: Basic product}
Let $k,m,n$ be positive integers, and let $w\in S_n$. Then:
    
    $$\mathfrak{S}_{w}\cdot h_m(x_1,\cdots , x_k)=\sum\limits_{v}\mathfrak{S}_{v},$$
where the sum is over all $v=wt_{a_1b_1}\cdots t_{a_mb_m}$ such that $a_i\leq k< b_i, a_i$ are in weak increasing and $\ell(wt_{a_1b_1}\cdots t_{a_ib_i})=\ell(w)+i$ for $1\leq i\leq m$ with the integers $b_1,\cdots, b_m$ distinct. Meanwhile, the condition $v=wt_{ab}$, $a\leq k<b$ and $\ell(v)=\ell(w)+1$ are equivalent to $w(a)<w(b)$ and there do not exist $a<i<b$ such that $w(a)<w(i)<w(b)$.
\end{lemma}

For simplicity, we also introduce the following defintions:
\begin{defin}
    Let $m,k$ be two positive integers and $w\in S_n$ be a permutation. For a permutation $u\in S_n,$ we denote $v\dashv h_{m}\mathfrak{S}_{w}(x_1,\cdots,x_k)$ if $\mathfrak{S}_{u}$ is one summation in the expression $\mathfrak{S}_{w}\cdot h_m(x_1,\cdots , x_k)=\sum\limits_{v}\mathfrak{S}_{v}$ as in Lemma~\ref{D: Basic product}.
    
    For any two permutations $u_1,u_2\dashv h_{m}\mathfrak{S}_{w},$ we say $u_1$ is \emph{equivalent} to $u_2$ and denote by $u_1\sim u_2$ if and only if there exists $u_1=wt_{a_{11} b_{11}}\cdots t_{a_{1m}b_{1m}}$, and $u_2=wt_{a_{21} b_{21}}\cdots t_{a_{2m}b_{2m}}$ satisfying the following conditions:
    \begin{itemize}
        \item For $i=1,2$ and $1\leq j\leq m,$ $\ell(wt_{a_{i1}b_{i1}}\cdots t_{a_{ij}b_{ij}})=\ell(w)+j$ with the integers $b_{i1},\cdots, b_{im}$ distinct;
        \item For $1\leq i\leq m,$ $b_{1i}=b_{2i}$;
        \item For any $1\leq i\leq k,$ $i$ appears continuously in $\{a_{21},\cdots, a_{2m}\}$ and $a_{1i}=a_{1(i+1)}$ if and only if $a_{2i}=a_{2(i+1)}.$
    \end{itemize}
\end{defin}
\begin{ex}
\label{D:example 2.9}
    Under the same assumption as in Example~\ref{D:example 2.5}, we have $h_2(x_1,x_2,x_3)\mathfrak{S}_{(1,4,3,2)}=\mathfrak{S}_{(1, 4,6,2,3,5)}+\mathfrak{S}_{(1, 6,3,2,4,5)}+\mathfrak{S}_{(2, 4, 5, 1, 3)}+\mathfrak{S}_{(2, 5, 3, 1, 4)}.$ It can checked directly that 
    \begin{align*}
        &(1,4,6,2,3,5)=(1,4,3,2)t_{3,5}t_{3,6};\\
        &(1, 6,3,2,4,5)=(1,4,3,2)t_{2,5}t_{2,6};\\
        &(2, 4, 5, 1, 3)=(1,4,3,2)t_{1,4}t_{3,5};\\
        &(2, 5, 3, 1, 4)=(1,4,3,2)t_{1,4}t_{2,5};
    \end{align*}
    The only two equivalent permutations are $(1,4,6,2,3,5)$ and $(1,6,3,2,4,5).$ 
\end{ex}

For two permutations $u_1\sim u_2\in S_n,$ we have the following lemma:
\begin{lemma}\label{D:lemma 2.8}
    Let $m_1,m_2$ be two positive integers. Suppose $u_1, u_2\dashv h_{m_1}\mathfrak{S}_{w}$ are two equivalent permutations, then there does not exist a permutation $u'$ such that $u'\dashv h_{m_2}\mathfrak{S}_{u_{1}}, h_{m_2}\mathfrak{S}_{u_{2}}$ simultaneously if $u_1\neq u_2.$
\end{lemma}
\begin{proof}
    By defintion, we suppose $u_1=wt_{a_{11} b_{11}}\cdots t_{a_{1{m_1}}b_{1{m_1}}}$, and $u_2=wt_{a_{21} b_{21}}\cdots t_{a_{2{m_1}}b_{2{m_1}}}$ as in Lemma~\ref{D: Basic product} such that $b_{1i}=b_{2i}$ and $a_{1i}=a_{1(i+1)}$ if and only if $a_{2i}=a_{2(i+1)},$ for $1\leq i\leq m_{1}$. If the lemma is not true, then there exists a permutation $v\dashv h_{m_2}\mathfrak{S}_{u_{1}},h_{m_2}\mathfrak{S}_{u_{2}}$ simultaneously. Since $\{u_{1}(i):~1\leq i\leq n\}=\{u_{2}(i):~1\leq i\leq n\}$ and $u_1\neq u_2$ there exists $i, j\leq k$ such that $i\neq j$ but $u_{1}(i)=u_{2}(j)$, denoted by $t$. We can suppose $i<j$ with $i\in\{a_{11},\cdots,a_{1{m_1}}\}$ and $j\in\{a_{21},\cdots,a_{2{m_1}}\}$ by definition of equivalence. By Lemma~\ref{D: Basic product}, we have $w(i)<u_{1}(i).$ Since $u_{2}(j)=u_{1}(i)>w(i),$ if we let $u'_{2}=wt_{a_{21}b_{21}}\cdots t_{a_{2c}b_{2c}}$ with $1\leq c\leq m$ and $a_{2c}<j=a_{2(c+1)},$ then $u'_{2}(j)=w(j)<u'_{2}(b_{2(c+1)})$ and $u_{2}(j)=u'_{2}(b_{2(c+1)}).$ Moreover, if we let $u'_{1}=wt_{a_{11}b_{11}}\cdots t_{a_{1d}b_{1d}}$ with $1\leq d\leq {m_1}$ and $a_{1d}<i=a_{1(d+1)},$ then $u'_{1}(i)=w(i)<u'_{1}(b_{1(d+1)})$ and for any $i<l<b_{1(d+1)},u'_{1}(i)=w(i)<u'_{1}(b_{1(d+1)})\leq u'_{1}(l)=w(l)$ or $u'_{1}(l)=w(l)\leq u'_{1}(i)=w(i)<u'_{1}(b_{1(d+1)}).$ In particular, we have $w(j)<w(i)$ since $w(j)<u_{2}(j)=u_1(i).$ Therefore $u_{2}(i)<u_{1}(i)=u_{2}(j).$ Therefore $v(i)<u_{2}(j)$ and $u_{1}(i)\leq v(i),$ which is a contradiction. Hence we finish our proof.
\end{proof}
\begin{ex}
    Take the same assumption as in Example~\ref{D:example 2.9}. Let $m_2$ be a positive integer. Let $u_1=(1,6,3,2,4,5)$ and $u_2=(1,4,6,2,3,5).$ There does not exsit a permutation $u'$ such that $u'\dashv h_{m_2}\mathfrak{S}_{u_1},h_{m_2}\mathfrak{S}_{u_2}.$ Otherwise, suppose there exists a permutation $v\dashv h_{m_2}\mathfrak{S}_{u_{1}},h_{m_2}\mathfrak{S}_{u_{2}}$ simultaneously. We have $u_1(2)=u_2(3)=6$ and $u_2(2)=4<u_1(2).$ Since $u_2(2)<u_2(3),$ we must have $v(2)<u_2(3)=6$. It contradicts to $v(2)\geq u_{1}(2)=6.$
\end{ex}

 
\section{Proof of theorem~\ref{I: Main theorem 1}}\label{sec::non}
We first give an example to show the upper bounds in Theorem~\ref{I: Main theorem 1} can be realized.
\begin{ex}
    Let $n_1=7,n_2=7.$ Then $n=7.$ Take $w=(1,2,3,5,7,4,6)$ to be a $5$-Grassmannian permutation in $S_7.$ Then $\lambda=(2,1).$ Let $\tau=(1,2,3,5,7,4,6).$ Then $\mathfrak{S}_{\tau}s_{\lambda}=\mathfrak{S}_{(1, 2, 3, 6, 9, 4, 5, 7, 8)}+\mathfrak{S}_{(1, 2, 3, 7, 8, 4, 5, 6)}+\mathfrak{S}_{(1, 2, 4, 5, 9, 3, 6, 7, 8)}+2*\mathfrak{S}_{(1, 2, 4, 6, 8, 3, 5, 7)}+\mathfrak{S}_{(1, 2, 5, 6, 7, 3, 4)}+\mathfrak{S}_{(1, 3, 4, 5, 8, 2, 6, 7)}+\mathfrak{S}_{(1, 3, 4, 6, 7, 2, 5)}.$  
\end{ex}

\begin{remark}
\label{S:remark}
Here we make some remarks on our problems and proof:
\begin{enumerate}
    \item It can be derived from intersection theory that the \emph{generalized Littlewood-Richardson coefficients} are nonnegative. Proving this assertion combinatorially and providing a way to compute these numbers is an important problem in algebraic combinatorics. Our result provides a combinatorical proof of the case $k=n_2$. The case $k=n_2-1$ and $k=n_2-2$ can also be done in a similar way. We omit it here for simplicity.
    \item The proof of theorems avoids the cases when $m_2$ is less than $n_2-k$. However, since $n_2-k$ is at most $2$, we can verify the theorem directly.
\end{enumerate}
\end{remark}
The main idea for the proof of Theorem~\ref{I: Main theorem 1} is to use Lemma~\ref{D:Jacobi-Trudi}. Thus we need to analyze the product of $h_{m_1}h_{m_2}$ with $\mathfrak{S}_{w}$ for any $m_1$,$m_2\in \mathbb{N}_{+}$ and $m_2\leq m_1.$  We first consider the case $k=n_2.$ We have the following lemma:
\begin{lemma}\label{S:lemma 3.2}
Let $k,n_2,n,m$ be positive integers and $n>n_2.$ Let $w\in S_{n_2},u\in S_n.$ 
If $k=n_2$ and $u\dashv h_{m}\mathfrak{S}_{w}(x_1,\cdots,x_k),$ then $u=wt_{a(n_2+1)}t_{a(n_2+2)}\cdots t_{a(n_2+m)}$ for some $1\leq a\leq n_2.$
\end{lemma}
\begin{proof}
By definition, we can suppose $u=wt_{a_{1}b_{1}}\cdots t_{a_{m}b_{m}},$ for some $a_{1}\leq a_{2}\leq\cdots\leq a_{m}\leq n$ and distinct $b_{i}>n,~1\leq i\leq m.$ Since $w(a_1)<w(n+1)=n+1<w(l)=n+l$ for any $l\geq 2,$ we have $b_1=n+1.$ if $a_{1}< a_{2},$ then $w(a_{2})\leq w(a_{1})<w(b_{1})$ or $w(a_{1})<w(b_{1})\leq w(a_{2}).$ Only the first case can happen. However, this means $wt_{a_{1}b_{1}}(a_2)<wt_{a_{1}b_{1}}(b_1)<wt_{a_{1}b_{1}}(b_{2}),$ which is a contradiction. Hence $a_1=a_2.$ Thus we finish our proof by induction. 
\end{proof}
\begin{ex}
    Let $m=2,k=3$ and $w=(1,3,2).$ Suppose $u\dashv h_{m}\mathfrak{S}_{w}.$ Then $u=(1,5,2,3,4)=wt_{2,4}t_{2,5}$ or $u=(1,3,4,2,5)=wt_{3,4}t_{3,5}.$
\end{ex}

\begin{prop}\label{S:prop 3.3}
Let $k$,$m_1$,$m_2\in \mathbb{Z}_{>0}$ and $w\in S_k$. Suppose $h_{m_1}h_{m_2}\mathfrak{S}_{w}(x_1,\cdots,x_k)=\sum c_v\mathfrak{S}_{v}$, then $c_{v}\in\{0,1\}$.
\end{prop}
\begin{proof}
By Lemma~\ref{D:lemma 2.8} and Lemma~\ref{S:lemma 3.2} we know that for any two permutations $u_1,u_2\dashv h_{m_2}\mathfrak{S}_{w},$ there does not exist a permutation $v\dashv h_{m_1}\mathfrak{S}_{u_1},h_{m_1}\mathfrak{S}_{u_2}$ simultaneously. Hence the assertion follows.
\end{proof}

Now we give the proof of Theorem~\ref{I: Main theorem 1}(1) for $k=n_2$:
\begin{proof}[Proof of Theorem~\ref{I: Main theorem 1}(1)]
    Suppose the Young diagram with $2$ rows $\lambda=(m_1,m_2).$ Then by Lemma~\ref{D:Jacobi-Trudi}, we have $\mathfrak{S}_{w}s_{\lambda}=h_{m_1}h_{m_2}\mathfrak{S}_{w}-h_{m_1+1}h_{m_2-1}\mathfrak{S}_{w}.$
    By Proposition~\ref{S:prop 3.3}, we have $c_{w\tau}^{v}\leq 1$ for any permutation $v\in S_n$. Now we show the coefficients are all nonnegative. 
    
    We first define two sets as follows:
\begin{align*}
    M_{1} &=\{v\in S_{n}:~ v\dashv h_{m_1}h_{m_2}\mathfrak{S}_{w}\}\\
    M_{2} &=\{v\in S_{n}:~ v\dashv h_{m_1-1}h_{m_2+1}\mathfrak{S}_{w}\}.
\end{align*}

We claim that $M_2\subseteq M_1$ and thus the coefficients are all nonnegative.

For any $v\in M_{2},$ suppose $v\dashv h_{m_1+1}\mathfrak{S}_{u}$ for $u=wt_{a(n_2+1)}t_{a(n_2+2)}\cdots t_{a(n_2+m_2-1)}\dashv h_{m_2-1}\mathfrak{S}_{w}$ satisfying the equation $v=ut_{a_1c_1}\cdots t_{a_{m_1+1}c_{m_1+1}}$. 

If there exists some $i$ satisfying $a_i=a,$ then we take $i$ to be the minimal number such that $a_i=a.$ We have $c_i=n_2+m_2.$ Let $u'=wt_{a(n_2+1)}t_{a(n_2+2)}\cdots t_{a(n_2+m_2-1)}t_{ac_{i}}$ and $v'=u^{'}t_{a_1c_1}\cdots t_{a_{i-1}c_{i-1}}t_{a_{i+1}c_{i+1}}\cdots t_{a_{m_1+1}c_{m_1+1}}=v.$ 

If such $i$ does not exist, there must be some $i_0$ such that $c_{i_0}=n_2+m_2.$ Otherwise, for each $i,$ $n_2<c_i<n_2+m_2,$ which contradicts to $m_2\leq m_1<m_1+1.$ Let $a'=a_{i_0}.$
If $a'>a$, we have $v=ut_{a'(n+2+1)}t_{a'(n_2+2)}\cdots t_{a'(n_2+m_1+1)}.$ Let $u'=wt_{a'(n_2+1)}t_{a'(n_2+2)}\cdots t_{a'(n_2+m_2)}$ and $v'=u't_{a(n_2+2)}t_{a(n_2+3)}\cdots t_{a(n_2+m_2)}t_{a'(n_2+m_2+1)}\cdots t_{a'(n_2+m_1+1)}=v.$ If $a'<a,$ we have  $v=ut_{a'(n_2+1)}t_{a'(n_2+2)}\cdots t_{a'(n_2+m_1+1)}$ or $v=ut_{a'(n_2+2)}t_{a'(n_2+3)}\cdots t_{a'(n_2+m_1+2)}.$ The first case can be done similarly. For the second case, let $u'=wt_{a'(n_2+1)}t_{a'(n_2+2)}\cdots t_{a'(n_2+m_2)}$ and $v=u't_{a(n_2+1)}t_{a(n_2+2)}\cdots t_{a(n_2+m_2)}t_{a(n_2+m_2+1)}\cdots t_{a(n_2+m_1)}.$ Thus we get $M_2\subseteq M_1.$ Hence we finish our proof.
\end{proof}
Here is an example to demonstrate the philosophy of our proof of the nonnegativity:
\begin{ex}
    Let $m_1=3,m_2=2,k=3$ and $w=(1,3,2).$ 
    Let $u=(1,4,2,3)=wt_{2,4}\dashv h_{1}(x_1,x_2,x_3)\mathfrak{S}_{w}.$ 
    
    For $v=(1,7,3,2,4,5,6)=ut_{2,5}t_{2,6}t_{2,7}t_{3,4},$ we have $u'=(1,5,2,3,4)=wt_{2,4}t_{2,5}$ and $v=u't_{2,6}t_{2,7}t_{3,4}.$

    For $v=(1,4,7,2,3,5,6)=ut_{3,4}t_{3,5}t_{3,6}t_{3,7},$ we have $u'=(1,3,5,2,4)=wt_{3,4}t_{3,5}$ and $v=u't_{2,5}t_{3,6}t_{3,7}.$
\end{ex}

Now we consider the case $k=n_2-1.$ We have the following lemma:
\begin{lemma}\label{S:lemma 3.5}
Let $k,m,n_2,n$ be positive integers and $n>n_2.$ Let $w\in S_{n_2},u\in S_n.$ If $k+1=n_2$ and $u\dashv h_{m}\mathfrak{S}_{w}(x_1,\cdots,x_k),$ then there are three equivalent classes of permutations of $u:$
\begin{itemize}
    \item Case $1$: $u=wt_{a(k+2)}t_{a(k+3)}\cdots t_{a(k+m+1)}$ for some $1\leq a\leq k$;
    \item Case $2$: $u=wt_{a(k+1)}t_{a(k+2)}\cdots t_{a(k+m)}$ for some $1\leq a\leq k$;
    \item Case $3$: $u=wt_{a_{1}(k+1)}t_{a_{2}(k+2)}\cdots t_{a_{2}(k+m)}$ for some $1\leq a_1\neq a_2\leq k.$ We do not require $a_1\leq a_2$ here.
\end{itemize}
\end{lemma}
\begin{proof}
Suppose $u\in S_n$ satisfies $u\dashv h_{m}\mathfrak{S}_{w}.$ By definition, we suppose that $u=wt_{a_{1}b_{1}}\cdots t_{a_{m}b_{m}},$ for some $a_{1}\leq a_{2}\leq\cdots\leq a_{m}\leq n$ and distinct $b_{i}>n,~1\leq i\leq m.$ If $a_1=\cdots=a_m,$ then  $u=wt_{a(k+2)}t_{a(k+3)}\cdots t_{a(k+m+1)}$ or $u=wt_{a(k+1)}t_{a(k+2)}\cdots t_{a(k+m)},$ since $w(i)$ is strictly increasing for $i\geq k+1=n.$ Otherwise, we must have $a_2=\cdots=a_m$ and $u=wt_{a_{1}(k+1)}t_{a_{2}(k+2)}\cdots t_{a_{2}(k+m)}$ by Lemma~\ref{D: Basic product}.
Hence we finish the proof.
\end{proof}
\begin{ex}
    Let $m=3,k=2$ and $w=(1,3,2).$ Suppose $u\dashv h_{3}\mathfrak{S}_{w}.$ Then $u=(1,6,2,3,4,5)=wt_{2,4}t_{2,5}t_{2,6}$ or $u=(5,3,1,2,4)=wt_{1,3}t_{1,4}t_{1,5}$ or $u=(2,5,1,3,4)=wt_{1,3}t_{2,4}t_{2,5}.$
\end{ex}

\begin{lemma}\label{S:lemma 3.6}
Let $m_1,m_2$ be two positive integers. Let $u_1$, $u_2\dashv h_{m_2}\mathfrak{S}_{w}$ be the permutations in Case $2,3$ repectively. Then there does not exist a permutation $v$ such that $v\dashv h_{m_1}\mathfrak{S}_{u_1}, h_{m_1}\mathfrak{S}_{u_2}$ simultaneously.
\end{lemma}

\begin{proof}
Suppose there exists a permutation $v$ such that $v\dashv h_{m_1}\mathfrak{S}_{u_1}, h_{m_1}\mathfrak{S}_{u_2}$ simultaneously. Suppose  $u_1=wt_{a(k+1)}t_{a(k+2)}\cdots t_{a(k+m_2)}$ and $u_2=wt_{a_{1}(k+1)}t_{a_{2}(k+2)}\cdots t_{a_{2}(k+m_2)}$. If $a\neq a_2,$ we suppose $a>a_2.$ Then $u_1(a_2)<u_2(a_2)=u_1(a).$ Thus $v(a_2)<u_1(a),$ which is a contradiction. Hence $a_2=a.$ Moreover, we have $u_1(a_1)=w(a_1), u_1(k+1)=w(a), u_1(k+2)=w(k+1)$ and $u_2(a_1)=w(k+1), u_2(k+1)=w(a_1), u_2(k+2)=w(a_2)=w(a).$ More explicitly, we have:
\begin{center}
\begin{tabular}{|c|c|c|c|}
\hline
     & $a_1$& $k+1$ & $k+2$ \\
     \hline
     $u_1$& $w(a_1)$& $w(a)$ &$w(k+1)$\\
     \hline
     $u_2$&$w(k+1)$&$w(a_1)$&$w(a)$\\
     \hline
\end{tabular}
\end{center}
From $u_1,$ we have $w(a)<w(k+1)$ and $w(k+1)<w(k+2)$ From $u_2,$ we have $w(a_1)<w(k+1).$ 

If $a_1<a_2,$ we have $w(a_2)<w(a_1)<w(k+1)$ and $w(a_2)=wt_{a_1(k+1)}(a_2)<wt_{a_1(k+1)}(k+2)=w(k+2)<wt_{a_1(k+1)}(k+1)=w(a_1).$ This is a contradiction. If $a_1>a_2,$ we have $w(k+1)<w(a_2)<w(k+2).$ This is also a contradiction. Hence we finish our proof.
\end{proof}
By the above lemma, we can obtain the following corollary directly:
\begin{cor}
Let $k$,$m_1$,$m_2\in \mathbb{Z}_{>0}$ and $w\in S_{k+1}$. Suppose $h_{m_1}h_{m_2}\mathfrak{S}_{w}=\sum c_v\mathfrak{S}_{v}$ in k variables, then $c_{v}\in\{0,1,2\}$.
\end{cor}

\begin{proof}[Proof of Theorem~\ref{I: Main theorem 1}(2)]
Suppose the Young diagram with $2$ rows $\lambda=(m_1,m_2).$ Then by Lemma~\ref{D:Jacobi-Trudi}, we have $\mathfrak{S}_{w}s_{\lambda}=h_{m_1}h_{m_2}\mathfrak{S}_{w}-h_{m_1+1}h_{m_2-1}\mathfrak{S}_{w}.$ To show $c^{v}_{w\tau}\leq 1,$ it suffices to consider $v\in h_{m_1}h_{m_2}\mathfrak{S}_{w}$ such that there are two different permutations $u_1,u_2\dashv h_{m_2}\mathfrak{S}_{w}$ and $v\dashv h_{m_1}\mathfrak{S}_{u_1},h_{m_2}\mathfrak{S}_{u_2}$ simultaneously. By lemma~\ref{D:lemma 2.8}, $u_1,u_2$ are in different cases in Lemma~\ref{S:lemma 3.5}.
By Lemma~\ref{S:lemma 3.6}, we can suppose $u_1$ is in Case $1.$ Let $ v=u_1t_{a_{1}b_1}t_{a_{2}b_2}\cdots t_{a_{m_1}b_{m_1}}$ and $u_1=wt_{a(k+2)}t_{a(k+3)}\cdots t_{a(k+m_2+1)}.$ Moreover, we suppose $v=u_2t_{a'_{1}b'_2}t_{a'_{2}b'_2}\cdots t_{a'_{m_1}b'_{m_1}}$ and $u_2=wt_{c(k+1)}t_{d(k+2)}\cdots t_{d(k+m_2)}.$ If $c=d,$ then $u_2$ is in Case $2.$ Otherwise it is in Case $3.$ 

If $d=a,$ we have $b_i\neq k+m_2+1$ for $1\leq i\leq m_1.$ Otherwise, suppose $b_{i_0}=k+m_2+1$ for $1\leq i_0\leq m_1,$ we have $a_{i_0}\neq a$ since $u_1(a)>u_1(k+m_2+1)=k+m_2.$ This contradicts to $u_2(d)=k+m_2.$ Therefore, we have $v\dashv h_{m_1+1}h_{m_2-1}\mathfrak{S}_{w}$ by $v=u_1't_{a(k+m_2+1)}t_{a_{1}b_1}t_{a_{2}b_2}\cdots t_{a_{m_1}b_{m_1}}$ and $u_1'=wt_{a(k+2)}t_{a(k+3)}\cdots t_{a(k+m_2)}.$

If $d\neq a,$ we have $d<a.$ Otherwise we have $u_2(a)<u_1(a)=w(k+m_2+1)$ and $u_2(d)=w(k+m_2).$ Thus $v(a)<u_2(d)$ from $u_2$ and $v(a)\geq u_1(a)$ from $u_1,$ which is a contradiction. If $c=d,$ we have $u_2(k+1)=w(d), u_2(d)=w(k+m_2), u_2(k+2)=w(k+1).$ More explicitly, we have:

\begin{center}
    \begin{tabular}{|c|c|c|c|c|}
\hline 
    & $d$ & $a$ & ${k+1}$ & ${k+2}$ \\
\hline
	${u_1}$ & ${w(d)}$ & ${w(k+m_2+1)}$ & ${w(k+1)}$ & ${w(a)}$ \\
 \hline
	${u_2}$ & ${w(k+m_2)}$ & ${w(a)}$ &$ {w(d)}$ & ${w(k+1)}$\\
 \hline
\end{tabular}
\end{center}

Thus, we have $w(a)<w(d)<w(k+1)<w(k+2)$ from $u_2$ and $w(k+1)<w(a)<w(k+2)$ from $u_1,$ which is a contradiction.

If $c\neq d,$ we have $u_2(k+1)=w(c),u_2(k+2)=w(d).$ More explicitly, we have:

\begin{center}
    \begin{tabular}{|c|c|c|c|c|}
\hline
	& $d$ & $a$ & ${k+1}$ & ${k+2}$ \\
 \hline
	${u_1}$ & ${w(d)}$ & ${w(k+m_2+1)}$ & ${w(k+1)}$ & ${w(a)}$ \\
 \hline
	${u_2}$ & ${w(k+m_2)}$ & ${w(a)}$ & ${w(c)}$ & ${w(d)}$\\
 \hline
\end{tabular}
\end{center}

Therefore, we have $w(c)<w(k+1)<w(a)<w(d)<w(k+2).$ This requires $v=u_2t_{a(k+2)}\cdots t_{a(k+m_1+1)}.$ and $v=u_1t_{c(k+1)}t_{d(k+3)}\cdots t_{d(k+m_2+1)}.$We have $v\dashv h_{m_1+1}h_{m_2-1}\mathfrak{S}_{w}$ by $v=u_2't_{c(k+1)}t_{a'_{1}b'_1}t_{a'_{2}b'_2}\cdots t_{a'_{m_1}b'_{m_1}}$ and $u_2'=wt_{d(k+2)}t_{d(k+3)}\cdots t_{d(k+m_2)}.$

Here we finish the proof.






\end{proof}

Now we consider the case $k=n_2-2.$ We have the following lemma similar to Lemma~\ref{S:lemma 3.5}:
\begin{lemma}
\label{S:lemma 3.99}
Let $k,m,n_2,n$ be positive integers and $n>n_2.$ Let $w\in S_{n_2},u\in S_n.$ If $k+2=n_2$ and $u\dashv h_{m}\mathfrak{S}_{w}(x_1,\cdots,x_k),$ then there are $12$ equivalent classes of permutations of $u:$
\begin{itemize}
    \item Case $1$: $u=wt_{a(k+3)}\cdots t_{a(k+m_2+2)}$ for some $1\leq a\leq k;$
    \item Case $2$: $u=wt_{a(k+1)}t_{a(k+3)}\cdots t_{a(k+m_2+1)}$ for some $1\leq a\leq k;$
    \item Case $3$: $u=wt_{a(k+2)}t_{a(k+3)}\cdots t_{a(k+m_2+1)}$ for some $1\leq a\leq k;$
    \item Case $4$: $u=wt_{a(k+1)}t_{a(k+2)}t_{a(k+3)}\cdots t_{a(k+m_2)}$
for some $1\leq a\leq k;$
\item Case $5$: $u=wt_{a(k+2)}t_{a(k+1)}t_{a(k+3)}\cdots t_{a(k+m_2)}$ for some $1\leq a\leq k;$
    \item Case $6$: $u=wt_{a_{1}(k+1)}t_{a_{2}(k+3)}\cdots t_{a_{2}(k+m_2+1)}$ for some $1\leq a_{1}\neq a_{2}\leq k;$
    \item Case $7$: $u=wt_{a_{1}(k+2)}t_{a_{2}(k+3)}\cdots t_{a_{2}(k+m_2+1)}$ for some $1\leq a_{1}\neq a_{2}\leq k;$
    \item Case $8$: $u=wt_{a_{1}(k+1)}t_{a_{2}(k+2)}t_{a_{2}(k+3)}\cdots t_{a_{2}(k+m_2)}$ for some $1\leq a_{1}\neq a_{2}\leq k;$
    \item Case $9$: $u=wt_{a_{1}(k+2)}t_{a_{2}(k+1)}t_{a_{2}(k+3)}\cdots t_{a_{2}(k+m_2)}$ for some $1\leq a_{1}\neq a_{2}\leq k;$
    \item Case $10$: $u=wt_{a_{1}(k+1)}t_{a_{2}(k+2)}t_{a_{3}(k+3)}\cdots t_{a_{3}(k+m_2)}$ for some $1\leq a_{1}\neq a_{2}\neq a_{3}\leq k;$
    \item Case $11$: $u=wt_{a_{1}(k+1)}t_{a_{1}(k+2)}t_{a_{2}(k+3)}\cdots t_{a_{2}(k+m_2)}$ for some $1\leq a_{1}\neq a_{2}\leq k;$
    \item Case $12$: $u=wt_{a_{1}(k+2)}t_{a_{1}(k+1)}t_{a_{2}(k+3)}\cdots t_{a_{2}(k+m_2)}$ for some $1\leq a_{1}\neq a_{2}\leq k.$\\
    
    \noindent We do not require $a_i$ are weakly increasing in all cases.
\end{itemize} 
\end{lemma}



\begin{prop}
\label{S:prop 3.10}
Let $k$,$m_1$,$m_2\in \mathbb{Z}_{>0}$ and $w\in S_{k+2}$. Suppose $h_{m_1}h_{m_2}\mathfrak{S}_{w}=\sum c_v\mathfrak{S}_{v}$ in k variables, then $c_{v}\in\{0,1,2,3,4,5\}$.
\end{prop}
\begin{proof}
    For $1\leq i\leq 12,$ let $u_i$ be a permutation such that $u_i\dashv h_{m_2}\mathfrak{S}_{w}(x_1,\cdots,x_k)$ from Case $i.$ Denote the `$a$' or `$a_j$' appears in the relation of $u_i$ and $w$ in Case $i$ by $a^i$ or $a^i_j.$
    
    Suppose there is a permutation $v\dashv h_{m_1}\mathfrak{S}_{u_1}.$ Then $v\not\dashv h_{m_1}\mathfrak{S}_{u_i}$ for $2\leq i\leq 5.$ Otherwise, suppose $v\dashv h_{m_1}\mathfrak{S}_{u_i}$ for some $2\leq i\leq 5.$ If $a^i=a^1,$ we have $w(k+1),w(k+2)<w(a^i)$ from $u_1$ since $w(k+1),w(k+2)<w(k+3)=k+3$ and $w(a^1)=w(a^i)<w(k+3).$ From $u_i,$ we have $w(k+1)<w(a^i)$ or $w(k+2)<w(a^i).$ It is a contradiction. If $a^i\neq a^1,$ we have $a^i<a^1.$ Otherwise, we have the $u_i(a^1)<u_1(a^1)=w(k+m_2+2)$ and $u_i(a^i)=w(k+m_2)$ or $w(k+m_2+1).$ Thus we have $v(a^1)<u_2(a^i)$ from $u_i$ and $v(a^1)\geq u_1(a^1)$ from $u_1.$ This is a contradiction. Now if $a_i<a,$ we have $w(a^1)<w(a^i)<w(k+1)<w(k+3)$ or $w(a^1)<w(a^i)<w(k+2)<w(k+3),$ which are all contradictions. Similarly, we can show $v\not\dashv h_{m_1}\mathfrak{S}_{u_i}$ for $i=8,9.$ Thus $c_v\leq 4.$

    Suppose there is a permutation $v\dashv h_{m_1}\mathfrak{S}_{u_2}.$ Then $v\not\dashv h_{m_1}\mathfrak{S}_{u_i}$ for $i=3,4.$ Otherwise, suppose $v\dashv h_{m_1}\mathfrak{S}_{u_i}$ for $i=3$ or $4.$ If $i=3,$ we have $a^3=a^2.$ We have $w(a^2)=w(a^3)<w(k+1)$ from $u_2$ and $w(a^2)=w(a^3)>w(k+1)$ from $u_3,$ which is a contradiction. If $i=4,$ we can show $a^4<a^2$ as above and with a similar argument, we can show there is a contradiction.
    
    Now we show that $v\not\dashv h_{m_1}\mathfrak{S}_{u_i}$ for $i=6,7.$ Otherwise, we have $a^i_2=a^2.$ For $i=6,$ we have $u_2(a^6_1)=w(a^6_1), u_2(k+1)=w(a^2), u_2(k+3)=w(k+1),u_6(a^6_1)=w(k+1),u_6(k+1)=w(a^6_1)$ and $u_6(k+3)=w(a^6_2).$ More explicitly, we have:
    \begin{center}
        \begin{tabular}{|c|c|c|c|}
        \hline
             & $a_1^6$ & $k+1$ &$k+3$ \\
        \hline
        $u_2$ & $w(a^6_1)$ & $w(a^2)$ &$w(k+1)$\\
        \hline
        $u_6$ & $w(k+1)$ & $w(a_1^6)$ &$w(a^6_2)$\\
        \hline
        \end{tabular}
    \end{center}
Therefore, $u_6=u_2t_{(k+1)(k+3)}t_{(a_1^6)(k+1)}.$ Since $w(k+2)<w(a^2)<w(k+1)$ and for any $a^6_1<j<k+1,$ $w(j)<w(a_1^6)<w(k+1),$ we have $\ell(u_6)=\ell(u_2)+2,$ which is a contradiction. For $i=7,$ we have $u_2(a^7_1)=w(a^7_1), u_2(k+1)=w(a^2), u_2(k+3)=w(k+1),u_7(a^7_1)=w(k+2),u_7(k+2)=w(a^7_1)$ and $u_7(k+3)=w(a^7_2).$ More explicitly, we have:
\begin{center}
\begin{tabular}{|c|c|c|c|c|}
\hline
    & ${a^7_1}$ & ${k+1}$ & ${k+2}$ &$ {k+3}$ \\
\hline
${u_2}$ & ${w(a^7_1)}$ & ${w(a^2)}$ & ${w(k+2)}$ & ${w(k+1)}$ \\
\hline
${u_7}$ & ${w(k+2)}$ &$ {w(k+1)} $& ${w(a_1^7)}$ & ${w(a^7_2)}$\\
\hline
\end{tabular}
\end{center} 
Therefore, $u_7=u_2t_{(k+1)(k+3)}t_{a_1^7(k+2)}$ and $\ell(u_7)=\ell(u_2)+2,$ which is a contradiction. 

Similarly, we can show if there is a permutation $v\dashv h_{m_1}\mathfrak{S}_{u_3}.$ Then $v\not\dashv h_{m_1}\mathfrak{S}_{u_i}$ for $i=2,5,6,7.$

For any $i,j\in\{4,5,8,9,10,11,12\}$ and $i<j,$ we have there does not exist a permutation $v\dashv h_{m_1}\mathfrak{S}_{u_i},h_{m_1}\mathfrak{S}_{u_j}$ except $\{i,j\}=\{10,11\}$ or $\{i,j\}=\{10,12\}.$ In fact, we have the following observation:
\begin{enumerate}
    \item Case $4$ implies $w(a^4)<w(k+1)<w(k+2)<w(k+3);$
    \item Case $5$ implies $w(a^5)<w(k+2)<w(k+1)<w(k+3);$
    \item Case $8$ implies $w(k+1)<w(a^8_2)<w(k+2)<w(k+3);$
    \item Case $9$ implies $w(k+2)<w(a^9_2)<w(k+1)<w(k+3);$
    \item Case $10$ implies $w(k+2),w(k+1)<w(a^{10}_3)<w(k+3);$
    \item Case $11$ implies $w(k+1)<w(k+2)<w(a^{10}_3)<w(k+3);$
    \item Case $12$ implies $w(k+2)<w(k+1)<w(a^{10}_3)<w(k+3);$
\end{enumerate} Then it is easy to see that there does not exist a permutation $v\dashv h_{m_1}\mathfrak{S}_{u_i},h_{m_1}\mathfrak{S}_{u_j}$ for any $i,j\in\{4,5,8,9,10,11,12\}$ and $i\leq 9<j.$

For any permutation $v\dashv h_{m_1}h_{m_2}\mathfrak{S}_{w},$ we define \[M=\{1\leq i\leq 12:~\text{there exists } u_i\dashv h_{m_2}\mathfrak{S}_{w} \text{ from Case } i \text{ such that } v\dashv h_{m_1}\mathfrak{S}_{u_i}\}.\] Then $|M|\leq 5$ and $|M|\geq 3$ if and only if $M\subseteq \{1,6,7,10,11\}$ or $M\subseteq \{1,6,7,10,12\}$

\end{proof}

\begin{proof}[Proof of Theorem~\ref{I: Main theorem 1}(3)]
Suppose the Young diagram with $2$ rows $\lambda=(m_1,m_2).$ Then by Lemma~\ref{D:Jacobi-Trudi}, we have $\mathfrak{S}_{w}s_{\lambda}=h_{m_1}h_{m_2}\mathfrak{S}_{w}-h_{m_1+1}h_{m_2-1}\mathfrak{S}_{w}.$ 
We only have to consider $v\dashv h_{m_1}h_{m_2}\mathfrak{S}_{w}$ with $c_v\geq 3,$ that is, the set $M$ defined in Proposition~\ref{S:prop 3.10} is contained in $\{1,6,7,10,11,12\}$ with $|M|\geq 3.$

For $\{1,6,7,10\}$ or $\{1,6,7,11\}$ or $\{1,6,7,11\}\subseteq M$, there exists a permutation $u_1\dashv h_{m_2}\mathfrak{S}_{w}$ from Case $1$ such that $v\dashv h_{m_1}\mathfrak{S}_{u_1}.$ Thus we suppose $v=u_1t_{a_{1}b_{1}}\cdots t_{a_{m_{1}}b_{m_{1}}}$ with $u_1=wt_{a(k+3)}\cdots t_{a(k+m_2+2)}.$ Moreover, there exists a permutation $u_6\dashv h_{m_2}\mathfrak{S}_{w}$ from Case $6$ such that $v\dashv h_{m_1}\mathfrak{S}_{u_6}.$ Thus we suppose $v=u_6t_{c_{1}d_{1}}\cdots t_{c_{m_{1}}d_{m_{1}}}$ with $u_6=wt_{a'_1(k+1)}t_{a'_2(k+3)}\cdots t_{a'_2(k+m_2+1)}.$ If $a_2'=a,$ we have $b_i\neq k+m_2+2$ for $1\leq i\leq m_1.$ Therefore, we have $v\dashv h_{m_1+1}h_{m_2-1}\mathfrak{S}_{w}$ by $v=u'_1t_{a(k+m_2+2)}t_{a_{1}b_{1}}\cdots t_{a_{m_{1}}b_{m_{1}}}$ with $u'_1=wt_{a(k+3)}\cdots t_{a(k+m_2+1)}.$ It is easy to see that $u_1'\dashv h_{m_2-1}\mathfrak{S}_{w}$ from Case $1.$ If $a_2'\neq a,$ we can show $a_2'<a$ similar as above. We have $a_1'\neq a.$ Then there does not exist $i$ such that $b_i=k+1.$ Otherwise, we have do transposition on positions $a_1$ and $k+2$ to get $v$ from $u_1.$ This requires we do transpositions $t_{a'(k+1)}t_{a'(k+2)}$ for some integer $a'\neq a_1$ to get $v$ from $u_2.$ In this way, we get $v(a_1)=w(k+1)$ from $u_2$ and $v(a_1)=w(k+2)$ from $u_1,$ which is a contradiction. Therefore, we have $v\dashv h_{m_1+1}h_{m_2-1}\mathfrak{S}_{w}$ by $v=u'_6t_{a_1'(k+1)}t_{c_{1}d_{1}}\cdots t_{c_{m_{1}}d_{m_{1}}}$ with $u'_6=wt_{a'_2(k+3)}\cdots t_{a'_2(k+m_2+1)}.$ It is easy to see that $u_1'\dashv h_{m_2-1}\mathfrak{S}_{w}$ from Case $1.$ 
Moreover, there exists a permutation $u_{10}\dashv h_{m_2}\mathfrak{S}_{w}$ from Case $10$ such that $v\dashv h_{m_1}\mathfrak{S}_{u_{10}}.$ We can similarly show that we have $v\dashv h_{m_1+1}\mathfrak{S}_{u'_6},h_{m_1+1}\mathfrak{S}_{u'_7}$ with $u'_6\dashv h_{m_2-1}\mathfrak{S}_{w}$ from Case $6$ and $u'_7\dashv h_{m_2-1}\mathfrak{S}_{w}$ from Case $7.$ Therefore $c^v_{w \tau}\leq 2.$

What is remained to prove is when $M=\{1,10,11\}$ or $M=\{1,10,12\}.$ We provide the proof of $M=\{1,10,11\}$ as an example. There exists a permutation $u_1\dashv h_{m_2}\mathfrak{S}_{w}$ from Case $1$ such that $v\dashv h_{m_1}\mathfrak{S}_{u_1}.$ Thus we suppose $v=u_1t_{a_{1}b_{1}}\cdots t_{a_{m_{1}}b_{m_{1}}}$ with $u_1=wt_{a(k+3)}\cdots t_{a(k+m_2+2)}.$ Moreover, there exists a permutation $u_{10}\dashv h_{m_2}\mathfrak{S}_{w}$ from Case $10$ such that $v\dashv h_{m_1}\mathfrak{S}_{u_{10}}.$ We can suppose $v=u_{10}t_{c_{1}d_{1}}\cdots t_{c_{m_{1}}d_{m_{1}}}$ with a permutation $u_{10}=wt_{a'_1(k+1)}t_{a'_2(k+2)}t_{a'_3(k+3)}\cdots t_{a'_3(k+m_2)}.$ If $a_3'=a,$ we have $b_i\neq k+m_2+2$ for $1\leq i\leq m_1.$ Therefore, we have $v\dashv h_{m_1+1}h_{m_2-1}\mathfrak{S}_{w}$ by $v=u'_1t_{a(k+m_2+2)}t_{a_{1}b_{1}}\cdots t_{a_{m_{1}}b_{m_{1}}}$ with $u'_1=wt_{a(k+3)}\cdots t_{a(k+m_2+1)}.$ It is easy to see that $u_1'\dashv h_{m_2-1}\mathfrak{S}_{w}$ from Case $1.$ If $a_3'\neq a,$ we have $a_3'<a$ and $a_1',a_2'\neq a.$ Therefore, we can also show $b_i\neq k+m_2+2$ for $1\leq i\leq m_1,$ since $v(u_{10})\geq u_1(a)=w(k+m_2+2).$
Hence we finish the proof.
\end{proof}

\section{Proof of theorem~\ref{I: Main theorem 2}}

For $k=n,n-1,n-2,$ we can see theorem~\ref{I: Main theorem 2} is true by theorem~\ref{I: Main theorem 1}.

We first define an indexed set to classify all equivalent classes of $u\dashv h_{m}\mathfrak{S}_{w}$ for any positive integer $m.$
\begin{defin}
    Let $n_2,k$ be two positive integers with $n_2\leq k.$ Define $M$ to be the set consists of $(i,j,I_1,I_2)$ such that
    \begin{itemize}
        \item $i,j$ are nonnegative integers such that $0\leq i\leq j\leq n_2-k.$
        \item If $i=0,$ $I_1=\varnothing.$ If $i>0$, $I_1=\{x_1,\cdots,x_t\}$ is a set of positive integers such that $\sum\limits_{m=1}^{k}x_m=i$ and $1\leq t\leq j$. 
        \item If $j=0,$ $I_2=\varnothing.$ If $j>0,$ $I_2=\{(1,y_1),\cdots, (j,y_j)\}$ is a set such that $\{y_1,\cdots,y_j\}\subseteq\{k+1,\cdots,n_2\}.$
    \end{itemize}
\end{defin}
\begin{ex}
    For $n_2=k,$ we have $M=\{(0,0,\varnothing,\varnothing)\}.$

    For $n_2=k-1,$ we have $M=\{(0,0,\varnothing,\varnothing),(0,1,\varnothing,\{(1,n_2-1)\}),(1,1,\{1\},\{(1,n_2-1)\})\}.$

    For $n_2=k-2,$ we have $M=\{(0,0,\varnothing,\varnothing),(0,1,\varnothing,\{(1,n_2-1)\}),(0,1,\varnothing,\{(1,n_2-2)\}),(0,2,\varnothing,\{(1,n_2-1),(2,n_2-2)\}),(0,2,\varnothing,\{(1,n_2-2),(2,n_2-1)\}),(1,1,\{1\},\{(1,n_2-1)\}),(1,1,\{1\},\{(1,n_2-2)\}),(1,2,\{1\},\{(1,n_2-2),(2,n_2-1)\}),(1,2,\{1\},(1,n_2-1),(2,n_2-2)),(2,2,\{2\},(1,n_2-1),(2,n_2-2)),(2,2,\{2\},(1,n_2-2),(2,n_2-1)),(2,2,\{1,1\},(1,n_2-1),(2,n_2-2)),(2,2,\{1,1\},(1,n_2-2),(2,n_2-1))\}.$
\end{ex}
\begin{defin}
\label{P:defff}
    Let $m,n_2,k$ be positive integers with $n_2>k.$ For any permutation $u\dashv h_{m}\mathfrak{S}_{w},$ we say $u$ is of Case $(i,j,I_1,I_2)\in M$ if and only if we have the relation
    $u=wt_{a_1b_1}\cdots t_{a_{1}b_{x_1}}t_{a_{2}b_{x_1+1}}\cdots t_{a_{t}b_{i}}t_{a_{t+1}b_{i+1}}\cdots t_{a_{t+1}b_{m}}$ such that
    \begin{itemize}
        \item For $1\leq l \leq t,$ $a_l $ appears $x_l $ times. 
        \item For $1\leq l  \leq j,$ $b_l =y_l .$ For $j<l \leq m,$ $b_l =n_2+l -j.$ 
        \item The expression satisfies conditions in Lemma~\ref{D:lemma 2.8} without requiring $a_i$ are in weak increasing.
    \end{itemize}
\end{defin}
\begin{ex}
    For $n_2=k,k-1$ it can be checked that the classification here coincides with that in Lemma~\ref{S:lemma 3.2} and Lemma~\ref{S:lemma 3.5}.

    For $n_2=k-2,$ we get a classification finer than that in Lemma~\ref{S:lemma 3.99}. The case $10$ in Lemma~\ref{S:lemma 3.99} consists of two cases here.
\end{ex}

Using a similar argument as in Lemma~\ref{S:lemma 3.2},~\ref{S:lemma 3.5} and ~\ref{S:lemma 3.99}, we can see this gives a complete classification of equivalent classes of permutations $u\dashv h_{m}\mathfrak{S}_{w}.$ In particular, if we further assume $w(k+1)>w(k+2)>\cdots>w(n_2),$ the set $\{u\dashv h_{m}\mathfrak{S}_{w}|~u \text{is in case }(i,j,I_{1},I_{2})\}$ is nonempty only if $I_{2}=\{(1,y_1),\cdots,(j,y_j)\}$ satisfies $y_1>y_2>\cdots>y_j.$
\begin{proof}[Proof of Theorem~\ref{I: Main theorem 2}]
    We define a partial order $\lesssim$ on $M$ by $(i,j,I_1,I_2)\lesssim (i',j',I_1',I_2')$ if and only if $j>j'.$ We say $(i,j,I_1,I_2)$ is \emph{less minimal} if there does not exist $(i,j,I_1,I_2)\lesssim (i',j',I_1',I_2')$ such that $j-j'\geq 2.$

    Suppose there is a permutation $v\dashv h_{m_1}\mathfrak{S}_{u}$ for $u\dashv h_{m_2}\mathfrak{S}_{w}$ from cases $J\subseteq M.$ Then same as proof in Theorem~\ref{I: Main theorem 1}, we claim that $v\dashv h_{m_1+1}\mathfrak{S}_{u'}$ for $u'\dashv h_{m_2-1}\mathfrak{S}_{w}$ from Cases $J'=\{(i,j,I_1,I_2)\in J|~(i,j,I_1,I_2) \text{ is not \emph{less minimal}}\}.$ Therefore we have $c_{\tau w}^{v}\leq |J|-|J'|.$ By direct calcultion, we have $|J|-|J'|\leq 2*\sum\limits_{j=1}^{n_2-k}\sum\limits_{x_1+\cdots+x_j=n_2-k}(n_2-k)!=2^{n_2-k}(n_2-k)!.$ If we further assume $w(k+1)>w(k+2)>\cdots>w(n_2),$ 
    there is no contribution of $(n_2-k)!$ and we have $|J|-|J'|\leq \sum\limits_{j=1}^{n_2-k}\sum\limits_{x_1+\cdots+x_j=n_2-k}1=2^{n_2-k}.$ 

    Now we provide the proof of the claim. Suppose $v\dashv h_{m_1}\mathfrak{S}_{u_1}$ for $u_1\dashv h_{m_2}\mathfrak{S}_{w}$ from Case $(i_1,j_1,I_1^1=\{x_1^1,\cdots,x_{t_1}^1\},I_2^1=\{(1,y_1^1),\cdots,(j_1,y_{j_1}^1)\})\in M$ and $v\dashv h_{m_1}\mathfrak{S}_{u_2}$ for $u_2\dashv h_{m_2}\mathfrak{S}_{w}$ from Case $(i_2,j_2,I_1^2=\{x_1^2,\cdots,x_{t_2}^2\},I_2^2=\{(1,y_1^2),\cdots,(j_2,y_{j_2})^2\})\in M.$ Let $j_1<j_2.$ By Definition~\ref{P:defff}, we can suppose $u_1=wt_{a^1_1b^1_1}\cdots t_{a^1_{1}b^1_{x^1_1}}t_{a^1_{2}b^1_{x^1_1+1}}\cdots t_{a^1_{t_1}b^1_{i_1}}t_{a^1_{t_1+1}b^1_{i_1+1}}\cdots t_{a^1_{t_1+1}b^1_{m_2}}$ and $u_2=wt_{a^2_1b^2_1}\cdots t_{a^2_{1}b^2_{x^2_1}}t_{a^2_{2}b^2_{x^2_1+1}}\cdots t_{a^2_{t_2}b^2_{i_2}}t_{a^2_{t_2+1}b^2_{i_2+1}}\cdots t_{a^2_{t_2+1}b^2_{m_2}}.$ Moreover, we suppose $v=u_1t_{c^1_1d^1_1}\cdots t_{c^1_{m_1}d^1_{m_1}}$ and $v=u_2t_{c^2_1d^2_1}\cdots t_{c^2_{m_1}d^2_{m_1}}$ satisfying conditions in Lemma~\ref{D: Basic product}.

    If $j_2-j_1\geq 2$ and $a_{t_2+1}^{2}=a_{t_1+1}^{1},$ we have $u_1(a_{t_1+1}^{1})=w(b^1_{m_2})=n_2+m_2-j_1$ and $u_2(a_{t_2+1}^{2})=w(b^2_{m_2})=n_2+m_2-j_2<u_1(a_{t_1+1}^{1}).$ Thus there does not exist $1\leq i\leq m_1$ such that $d^1_{i}=n_2+m_2-j_1.$ Otherwise, if some $i_0$ such that $d^1_{i_0}=n_2+m_2-j_1,$ then $c^1_{i_0}\neq a^1_{t_1+1}$ since $u_1(a_{t_1+1}^{1})>u_1(d^1_{i_0})=n_2+m_2-j_1-1.$ However, since $u_2(a_{t_2+1}^{2})=n_2+m_2-j_2<u_1(d^1_{i_0})<u_1(a_{t_1+1}^{1}),$ we must do transpositions on $u_2$ with positions $a_{t_2+1}^2$ and $d^1_{i_0}.$ This means there does not exist $i\neq a_{t_2+1}^2=a_{t_1+1}^1$ such that $v(i)\geq n_2+m_2-j_1-1,$ which is a contradiction. Thus there does not exist $1\leq i\leq m_1$ such that $d^1_{i}=n_2+m_2-j_1.$ We have $v=u_1't_{a^1_{t_1+1}b^1_{m_2}}t_{c^1_1d^1_1}\cdots t_{c^1_{m_1}d^1_{m_1}}$ with $u_1'=u_1=wt_{a^1_1b^1_1}\cdots t_{a^1_{1}b^1_{x^1_1}}t_{a^1_{2}b^1_{x^1_1+1}}\cdots t_{a^1_{t_1}b^1_{i_1}}t_{a^1_{t_1+1}b^1_{i_1+1}}\cdots t_{a^1_{t_1+1}b^1_{m_2-1}}.$ It is easy to see $u_1\dashv h_{m_2-1}\mathfrak{S}_{w}$ from Case $(i_1,j_1,I_1^1,I_2^1).$

    If $j_2-j_1\geq 2$ and $a_{t_2+1}^{2}\neq a_{t_1+1}^{1},$ we have $a_{t_2+1}^{2}< a_{t_1+1}^{1}.$ Otherwise, we have $u_1(a_{t_1+1}^{1})>u_2(a_{t_2+1}^{2})>u_2(a_{t_1+1}^{1})$ and $a_{t_1+1}^{1}<a_{t_2+1}^{2}.$ From $u_1,$ we have $v(a_{t_1+1}^{1})\geq u_1(a_{t_1+1}^{1}).$ From $u_2,$ we have $v(a_{t_1+1}^{1})< u_2(a_{t_2+1}^{2}).$ This a contradiction. Thus $a_{t_2+1}^{2}< a_{t_1+1}^{1}.$ We can show there does not exist $1\leq i\leq m_1$ such that $d^1_{i}=n_2+m_2-j_1$ with a similar argument as above and get $v\dashv h_{m_1+1}\mathfrak{S}_{u_1'}$ for some $u_1'\dashv h_{m_2-1}\mathfrak{S}_{w}$ from Case $(i_1,j_1,I_1^1,I_2^1).$
\end{proof}
\begin{remark}
    With the same notation in the proof of Theorem~\ref{I: Main theorem 2},  we can show if $j_2-j_1=1$ and $a_{t_2+1}^2=a_{t_1+1}^{1},$ we have $v\dashv h_{m_1+1}\mathfrak{S}_{u_1'}$ for some $u_1'\dashv h_{m_2-1}\mathfrak{S}_{w}$ from Case $(i_1,j_1,I_1^1,I_2^1).$ However, when $a_{t_2+1}^2\neq a_{t_1+1}^{1},$ it is too complex.
\end{remark}

We obtain Conjecture~\ref{I: Conj 3} with the help of sagemath.
It can be checked easily that Conjecture~\ref{I: Conj 3} holds for $n_2-k=0,1,2$. Moreover, we give an example for $n_2-k=5$:
\begin{ex}
    Let $n_1=10,n_2=11,k=6$ and $\tau=(1,2,3,4,8,10,5,6,7,9)$ be a $6$-Grassmannian. Then $\lambda=(4,3).$ Let $w=(6,5,4,3,2,1,11,10,9,8,7).$ We have $\mathfrak{S}_{w}s_{\lambda}=\sum\limits_{v\in S_{11}}c_{ w\tau}^{v}\mathfrak{S}_{v}$ with $38194$ summations and the largest $c_{w\tau}^{v}=5.$ 
\end{ex}

\section*{Acknowledgements}
We are grateful to Yibo Gao who brought the topic and extended several enlightened talks to us. We also give special thank to Dr. Tianyi Yu who give precious advice to us. 
\bibliographystyle{plain}
\bibliography{ref}

\end{document}